\documentclass[2p]{article}
\usepackage{amssymb,amsmath,amsthm}
\usepackage{indentfirst}
\usepackage{exscale}
\usepackage{relsize}
\usepackage{subfigure}
\usepackage{graphicx}
\usepackage{epstopdf}
\usepackage{xcolor}
\usepackage[numbers,sort&compress]{natbib}

\usepackage{geometry}
\newcommand{\R}{\mathbb{R}}
\newtheorem{theorem}{Theorem}[section]

\theoremstyle{definition}
\newtheorem{definition}[theorem]{Definition}

\newtheorem{lemma}[theorem]{Lemma}

\newtheorem{corollary}[theorem]{Corollary}

\allowdisplaybreaks
\theoremstyle{remark}

\numberwithin{equation}{section}
\usepackage{amsmath}
\numberwithin{equation}{section}
\UseRawInputEncoding

\begin{document}
\title{\Large\bf{ Least energy sign-changing solutions for asymptotically cubic Kirchhoff equations on locally finite graphs}}
\date{}
\author {Shuai An$^1$, \ Junping Xie$^{2}$\footnote{Corresponding author, E-mail address: hnxiejunping@163.com},\ \  Xingyong Zhang$^{1,2}$\\
{\footnotesize $^1$Faculty of Science, Kunming University of Science and Technology,}\\
 {\footnotesize Kunming, Yunnan, 650500, P.R. China.}\\
 {\footnotesize $^2$Faculty of Transportation Engineering, Kunming University of Science and Technology,}\\
 {\footnotesize Kunming, Yunnan, 650500, P.R. China.}\\
{\footnotesize $^{3}$Research Center for Mathematics and Interdisciplinary Sciences, Kunming University of Science and Technology,}\\
 {\footnotesize Kunming, Yunnan, 650500, P.R. China.}\\
}
 \date{}
 \maketitle

 \begin{center}
 \begin{minipage}{15cm}
 \small  {\bf Abstract:} We obtain the existence results of least energy sign-changing solutions and ground state solutions for a class of asymptotically cubic Kirchhoff equations with Dirichlet boundary value on a locally finite graph $G=(V,E)$, and obtain the least energy of sign-changing solutions is strictly larger than twice of the ground state energy.
\par
{\bf Keywords:}  asymptotically cubic Kirchhoff equations, locally finite graphs, least energy sign-changing solutions, ground state solutions, non-Nehari manifold method.
 \end{minipage}
 \end{center}
  \allowdisplaybreaks
\vskip2mm
{\section{Introduction }}
\setcounter{equation}{0}
First, we recall some basic concepts and notations on graphs (\cite{Grigor 2017} and \cite{Grigor 2016}). Let $G=(V,E)$ be a graph, where $V$ and $E$ denote the vertex set and the edge set, respectively. The graph $G$ is called locally finite if each vertex $x\in V$ has only finitely many neighbors $y\in V$ satisfying $xy\in E$. Moreover, $G$ is connected if any two vertices can be joined by a path consisting of finitely many edges. We write $y\thicksim x$ whenever $xy\in E$. For each edge $xy\in E$, we assign a positive weight $\omega_{xy}$ satisfying
$
\omega_{xy}=\omega_{yx},\omega_{xy}>0.
$
The weighted degree of a vertex $x\in V$ is defined by $deg(x)=\sum_{y\thicksim x}\omega_{xy}.$
For any two vertices $x,y\in V$, the graph distance $d(x,y)$ is defined as the minimum number of edges in a path connecting $x$ and $y$. Let $\Omega\subset V$. If $d(x,y)$ is uniformly bounded for all $x,y\in\Omega$, then $\Omega$ is called a bounded domain in $V$. The boundary of $\Omega$ is defined by
\[
\partial\Omega=\{y\in V,\;y\notin\Omega|\;\exists\;x\in\Omega\;\text{such that}\;xy\in E\},
\]
and the interior of $\Omega$ is given by
$\Omega^{\circ}=\Omega\backslash\partial\Omega.$
It is easy to see that
$\Omega=\Omega^\circ,$
which is different from the corresponding situation in the Euclidean setting.

Assume that $\mu:V\rightarrow \mathbb{R}^{+}$ is a finite positive function.
The gradient form associated with the graph is defined by
\begin{eqnarray}
\label{eq2}
\Gamma(\psi_1,\psi_2)(x)
=
\frac{1}{2\mu(x)}
\sum\limits_{y\thicksim x}
w_{xy}(\psi_1(y)-\psi_1(x))
(\psi_2(y)-\psi_2(x))
:=\nabla\psi_1\cdot\nabla\psi_2 .
\end{eqnarray}
For $\Gamma(\psi,\psi)$, we simply write
$
\Gamma(\psi)=\Gamma(\psi,\psi),
$
and the length of the gradient is defined as
\begin{eqnarray}
\label{eq3}
|\nabla\psi|(x)
=
\sqrt{\Gamma(\psi)(x)}
=
\left(
\frac{1}{2\mu(x)}
\sum\limits_{y\thicksim x}
w_{xy}(\psi(y)-\psi(x))^2
\right)^{\frac12}.
\end{eqnarray}
For a function $\psi:\Omega\rightarrow\mathbb{R}$, the Laplacian operator $\Delta$ is given by
\begin{eqnarray}
\label{eq1}
\Delta\psi(x)
=
\frac{1}{\mu(x)}
\sum\limits_{y\thicksim x}
w_{xy}(\psi(y)-\psi(x)).
\end{eqnarray}
The integral with respect to the measure $\mu$ is defined by
\begin{eqnarray}
\label{eq4}
\int_{\Omega}\psi(x)d\mu
=
\sum\limits_{x\in\Omega}\mu(x)\psi(x).
\end{eqnarray}
In the distributional sense, the Laplacian operator satisfies
\begin{eqnarray*}
\int_{\Omega\cup\partial\Omega}(\Delta\psi)\phi d\mu
=
-\int_{\Omega\cup\partial\Omega}\Gamma(\psi,\phi)d\mu,
\end{eqnarray*}
for any $\psi\in C_c(\Omega)$, where $C_c(\Omega)$ denotes the space of real-valued continuous functions with compact support.
The Sobolev space $\mathcal{H}_{0}^{1,2}(\Omega)$ is defined by
\begin{eqnarray}
\label{b1}
\mathcal{H}_{0}^{1,2}(\Omega)
=
\left\{u\Big|
u:\Omega\rightarrow\mathbb{R},
\;
u|_{\partial\Omega}=0
\right\}.
\end{eqnarray}
It is equipped with the norm
\begin{eqnarray*}
\|u\|
=
\left(
\int_{\Omega\cup\partial\Omega}
|\nabla u(x)|^2d\mu
\right)^{\frac12}.
\end{eqnarray*}
Then $\mathcal{H}_{0}^{1,2}(\Omega)$ is a finite-dimensional Banach space.
For $2\leq r<+\infty$, the space $L^r(\Omega)$ is defined as
\[
L^r(\Omega)
=
\left\{u\Big|
u:\Omega\rightarrow\mathbb{R}
\right\},
\]
with the norm
\[
\|u\|_{L^r(\Omega)}
=
\left(
\int_{\Omega}|u(x)|^r d\mu
\right)^{\frac1r}.
\]
When $r=+\infty$, we define
\[
L^\infty(\Omega)
=
\left\{
u:\Omega\rightarrow\mathbb{R}:
\sup_{x\in\Omega}|u(x)|<\infty
\right\},
\]
and its norm is given by
\[
\|u\|_{L^\infty(\Omega)}
=
\max_{x\in\Omega}|u(x)|.
\]
\par
In recent years, partial differential equations on locally finite graphs have attracted increasing interest due to their applications in image processing, data analysis, and machine learning \cite{Cheung 2018,Chung 2005,Ta 2011,Elmoataz 2012}.
The existence of solutions for such equations has been extensively studied, and a variety of results have been obtained; see, for example, \cite{Pan 2023,Ou 2024,Yu 2025,Grigor 2017,Grigor 2016,Han 2021,Qiu 2023,Shao 2024,Zhang 2018}. A common feature of the existing results is that the nonlinear terms are usually required to satisfy suitable super-quartic growth conditions. For example, Pan-Ji \cite{Pan 2023} investigated the least energy sign-changing solutions for the nonlinear Kirchhoff equation
\[
-\left(a+b\int_V|\nabla u|^2\,d\mu\right)\Delta u
+c(x)u=f(u),
\qquad x\in V,
\]
on a locally finite graph $G=(V,E)$, where $a,b>0$ and $c:V\to\mathbb{R}$. By means of a constrained variational method, they obtained the existence of a least energy sign-changing solution. More precisely, their result was established under the assumptions that the nonlinear function $f$ satisfies the following conditions:

\noindent
{\it \(\mathrm{(C_1)}\)
\(\displaystyle \frac{F(u)}{u^4}\to+\infty\), as \(|u|\to\infty\), where
\(F(u)=\displaystyle\int_0^u f(s)\,ds\),}
and

\noindent
{\it \(\mathrm{(C_2)}\)
\(\displaystyle \frac{f(u)}{u}\to0\), as \(u\to0\).}

\noindent
Furthermore, they proved that the least energy level of sign-changing solutions is strictly larger than twice the least energy. Subsequently, Ou-Zhang \cite{Ou 2024} considered a more general class of Kirchhoff-type equations involving power-law terms, logarithmic nonlinearities, and Dirichlet boundary conditions on locally finite graphs. The problem they studied is given by
\[
\begin{cases}
-\left(
a+b\displaystyle\int_{\overline{\Omega}}|\nabla u|^2\,d\mu
\right)\Delta u
+\lambda g(x)u^{\frac{2k}{m}-1}
=
Q(x)|u|^{p-2}u\ln |u|^r,
& x\in\Omega^\circ,\\
u=0,
& x\in\partial\Omega,
\end{cases}
\]
where $\overline{\Omega}\subset V$ is a bounded domain,
$
\overline{\Omega}:=\Omega^\circ\cup\partial\Omega
=\Omega\cup\partial\Omega,
$
$a>0$, $b\geq0$, $p>4$, $\lambda\geq0$, $r\geq1$, and $m,k\in\mathbb{N}$ satisfy
$
1<\frac{2k}{m}\leq p,
$
with
$
Q,g\in C\bigl(\Omega^\circ,(0,\infty)\bigr).
$
They also established a corresponding existence result for least energy sign-changing solutions. However, both works \cite{Pan 2023,Ou 2024} rely on the super-quartic growth assumption $\mathrm{(C_1)}$ at infinity for the nonlinear term.

Recently, Kirchhoff equations on locally finite graphs with different types of nonlinearities have attracted considerable attention. In particular, Yu-Zhang-Ou \cite{Yu 2025} studied the least energy solutions of two Kirchhoff equations with asymptotically cubic nonlinear terms on a weighted, connected, and locally finite graph $G=(V,E)$. The nonlinear term appearing in their models is
$
\lambda u+\eta |u|^2u,
$
which has asymptotically cubic growth at infinity and satisfies neither $\mathrm{(C_1)}$ nor $\mathrm{(C_2)}$.
The two Kirchhoff equations studied in their work are formulated as follows:
\[
\begin{cases}
-\left(a+b\displaystyle\int_{\Omega}|\nabla u(x)|^2\,d\mu\right)\Delta u(x)
=\lambda u(x)+\eta|u(x)|^2u(x),
& x\in\Omega,\\[1mm]
u(x)=0,
& x\in\partial\Omega.
\end{cases}
\]
and
\[
-\left(
a+b\int_V\left(|\nabla u(x)|^2+h(x)u^2(x)\right)d\mu
\right)
\left(\Delta u(x)+h(x)u(x)\right)
=
\lambda u(x)+\eta|u(x)|^2u(x),
\quad x\in V.
\]
Here, $\Omega\cup\partial\Omega\subset V$ is a bounded domain,
$a>0$, $b\geq0$, $\lambda,\eta\in\mathbb{R}$, $u:V\to\mathbb{R}$, and $h:V\to\mathbb{R}$.
Using the constrained variational method, they proved the existence of least energy solutions for both equations. More precisely, there exist constants $\lambda_1\geq0$ and $\eta_0\geq0$ (respectively, $\lambda_1^*\geq0$ and $\eta_0^*\geq0$) such that the first equation admits at least one least energy solution whenever
$
|\lambda|<a\lambda_1, \mbox{ and } \eta>\eta_0,
$
while the second equation admits at least one least energy solution whenever
$
|\lambda|<a\lambda_1^*, \mbox{ and }  \eta>\eta_0^*.
$
However, the existence of least energy sign-changing solutions was not investigated in \cite{Yu 2025}. This is mainly due to the fact that, on locally finite graphs, the decomposition
$
\|u\|^2=\|u^+\|^2+\|u^-\|^2
$
contains additional terms compared with the Euclidean setting; see Proposition 3.2 in \cite{Ou 2024}. Consequently, it is difficult to obtain results similar to those in \cite{Li 2022} and \cite{Khoutir 2021} where the existence of least energy sign-changing solutions for Kirchhoff equations and Schr\"odinger-Poisson systems was studied in the continuous setting. Thus, a natural question is whether the existence result on the least energy sign-changing solutions for Kirchhoff equations on locally finite graphs can be obtained  when the nonlinear term satisfies asymptotically cubic condition at infinity?

\par
In this paper, we focus on  the following Kirchhoff-type equation with Dirichlet boundary value condition on a locally finite graph $G=(V,E)$:
\begin{eqnarray}
\label{eq1}
 \begin{cases}
  -\left(a+b\displaystyle\int_{\overline{\Omega}}|\nabla u|^2d\mu\right)\Delta u=f(u),& \text {in} \; \Omega^\circ,\\
  u=0,&\text {on} \; \partial\Omega,\\
   \end{cases}
\end{eqnarray}
where $\overline{\Omega} \subset V $ is a bounded domain, $\overline{\Omega}:=\Omega^\circ\cup\partial \Omega=\Omega\cup\partial \Omega$, $a>0$, $b\geq0$. Throughout this paper, we make the following assumptions:

{\it \noindent $(F_1)$ $f\in C(\mathbb{R},\mathbb{R})$, and there exist constants $C_0>0$ and $p\in[4,+\infty)$ such that
\[
|f(t)|\le C_0\left(1+|t|^{p-1}\right), \forall\, t\in\mathbb{R};
\]

\noindent $(F_2)$ $f(t)=o(t)$ as $|t|\to0$;

\noindent $(F_3)$ $f(t)=lt^3+f_\infty(t)$, where $l>0$ and $f_\infty(t)$ satisfies

\noindent $(I)$ $f_\infty(t)=o(t^3)$ as $|t|\to\infty$, and

\noindent $(II)$ there exists $\theta_0\in(0,1)$ such that for all
$t>0$, and $\tau\in\mathbb{R}\setminus\{0\}$,
\[
\left[
\frac{f_\infty(\tau)}{\tau^3}
-
\frac{f_\infty(t\tau)}{(t\tau)^3}
\right]\operatorname{sign}(1-t)
+
\frac{a\theta_0\lambda_1|1-t^2|}{(t\tau)^2}
\ge0,
\]
where $\lambda_1$ is the first eigenvalue of $(-\Delta,\mathcal{H}_0^{1,2}(\Omega))$.

\noindent $(F_4)$ $l>b\mu_0$, where
\[
\mu_0
=
\inf_{\substack{u\in \mathcal{H}_0^{1,2}(\Omega)\\ u^+\neq0,\;u^-\neq0}}
\max\left\{
\frac{\|u\|^2\displaystyle\int_{\Omega\cup\partial\Omega}\Gamma(u,u^+)\,d\mu}
{\displaystyle\int_{\Omega}|u^+|^4\,d\mu},
\
\frac{\|u\|^2\displaystyle\int_{\Omega\cup\partial\Omega}\Gamma(u,u^-)\,d\mu}
{\displaystyle\int_{\Omega}|u^-|^4\,d\mu}
\right\}.
\]
and
$$
u^+(x):=\max\{u(x),0\}\;\;\text{and}\;\; u^-(x):=\min\{u(x),0\}.
$$
}
It is easy to see that $(F_3)$ implies that the nonlinear term $f$ satisfies the  asymptotically cubic condition at infinity. We shall establish the existence results of least energy sign-changing solutions and least energy solutions for problem (\ref{eq1}) by using the method and tools developed in \cite{Cheng 2017} and \cite{Ou 2024}, and show that the least energy of sign-changing solutions is strictly larger than twice the ground state energy.

\par
We have to address that the existence and energy properties of solutions to Kirchhoff-type equations in the continuous setting have been extensively studied. For the related works in this direction, we refer readers to Refs. \cite{Feng 2021,Xie 2020,Shuai 2015,Cheng 2017,Tang 2016,H 2015,Sun 2019,Han 2021}. Especially, our work is inspired by \cite{Cheng 2017} greatly, where Cheng-Tang studied the least energy sign-changing solutions to the following asymptotically cubic Kirchhoff-type problem:
\[
\begin{cases}
-\left(
a+b\displaystyle\int_{\Omega}|\nabla u|^2\,dx
\right)\Delta u=f(x,u),
& x\in\Omega,\\[1mm]
u=0,
& x\in\partial\Omega.
\end{cases}
\]
Here $\Delta$ denotes the classical Laplacian operator, that is, $\Delta=\sum_{i=1}^N\frac{\partial^2}{\partial x^2}$,  \(\Omega\subset\mathbb R^N\) is a bounded domain with smooth boundary, \(N=1,2,3\) and \(a,b>0\). They used the non-Nehari manifold method developed by \cite{Shuai 2015} and \cite{Tang 2016} to prove the existence of a ground-state sign-changing solution when \(f\) is assumed to satisfy the following conditions:

\noindent
{\it\(\mathrm{(H_1)}\)
\(f \in C(\Omega \times \mathbb{R},\mathbb{R})\), and there exist constants \(C_0>0\) and \(p\in[4,2^*)\) such that:
\[
|f(x,t)|\leq C_0(1+|t|^{p-1}), \forall (x,t)\in \Omega\times\mathbb{R},
\]
where \(2^*=6\) if \(N=3\), and \(2^*=+\infty\) if \(N=1,2\);}

\noindent
\(\mathrm{(H_2)}\)
\(f(x,t)=o(t)\) uniformly in \(x\in\Omega\) as \(|t|\to0\);

\noindent
{\it\(\mathrm{(H_3)}\)
\(f(x,t)=lt^3+f_\infty(x,t)\), where \(l>0\) and
\(f_\infty(x,t)\) satisfies

\noindent
\(\mathrm{( I )}\)
\(f_\infty(x,t)=o(t^3)\) uniformly in \(x\in\Omega\) as
\(|t|\to\infty\), and

\noindent
\(\mathrm{(II)}\)
there exists a \(\theta_0\in(0,1)\) such that for all \(x\in\Omega\), \(t>0\), and
\(\tau\in\mathbb R\setminus\{0\}\)
\[
\begin{aligned}
\left[
\frac{f_\infty(x,\tau)}{\tau^3}
-
\frac{f_\infty(x,t\tau)}{(t\tau)^3}
\right]
\operatorname{sign}(1-t)
+
\frac{a\theta_0\lambda_1|1-t^2|}{(t\tau)^2}
\geq0,
\end{aligned}
\]
where \(\lambda_1\) is the first eigenvalue of
the Laplacian operator $(-\Delta, H_0^{1}(\Omega))$, and}
$$
H_0^{1}(\Omega)=\left\{u:\Omega\to\R| \int_\Omega |u|^2dx<\infty, \int_\Omega |\nabla u|^2dx<\infty\right\}.
$$

\noindent
They also showed  that the sign-changing least energy is strictly larger than twice that of the least energy. It is worth noting that the condition \((H_3)\) is much weaker than the standard  Nehari-type monotonicity condition in \cite{Shuai 2015}. One of the purpose in our paper is to develop the results in \cite{Cheng 2017} from the continuous setting to the locally finite graph setting. Different from the continuous setting, the graph Laplacian operator is determined by the interactions between adjacent vertices, which makes some analysis different from that in the Euclidean setting.

\par
To present our results clearly, we define the functional $I:\mathcal{H}_{0}^{1,2}(\Omega)\to\R$ as
\begin{eqnarray}
\label{b3}
    I(u)
&=& \frac{a}{2}\int_{\Omega\cup \partial \Omega} |\nabla u|^2d\mu
       +\frac{b}{4}\left(\int_{\Omega\cup \partial \Omega} |\nabla u|^2d\mu\right)^2
 -\int_{\Omega}F(u)d\mu,
\end{eqnarray}
where $F(u)=\int_0^uf(t)dt$. A standard argument implies that $I\in C^1(\mathcal{H}^{1,2}_0(\Omega),\R)$ (for example, see \cite{Yang 2024}) and
\begin{eqnarray}
\label{b4}
    \langle I'(u),v\rangle
&=& a\int_{\Omega\cup \partial \Omega}\nabla u\cdot\nabla vd\mu
      +b\int_{\Omega\cup \partial \Omega} |\nabla u|^2d\mu\int_{\Omega\cup \partial \Omega}\nabla u\cdot\nabla v d\mu-\int_{\Omega}f(u)vd\mu
\end{eqnarray}
for any $u, v\in \mathcal{H}^{1,2}_0(\Omega)$.

\par
\vskip2mm
\noindent
{\bf Definition 1.1.} {\it If $u\in \mathcal{H}^{1,2}_0(\Omega)$ is a solution of equation (\ref{eq1}) and $u^\pm\not\equiv 0$, then $u$ is a sign-changing solution of (\ref{eq1}).}

\vskip2mm
\noindent

 Next, we present our main result.
 \vskip2mm
\noindent
{\bf Theorem 1.2.} {\it Assume that $G=(V,E)$ is a locally finite graph satisfying
$\Omega^\circ\neq\emptyset$ and $\partial\Omega\neq\emptyset$.
Then equation (\ref{eq1}) has a sign-changing solution $\tilde{u}\in\mathcal{M}$ and a constant-sign solution $\bar{u}\in\mathcal{N}$ such that $
I(\tilde{u})=\inf_{\mathcal{M}}I=:m,
I(\bar{u})=\inf_{\mathcal{N}}I=:c,
$
where
\[
\mathcal{M}=\left\{
u\in \mathcal{H}^{1,2}_0(\Omega),\;u^\pm\neq0,\;
\langle I'(u),u^+\rangle=0,\;
\langle I'(u),u^-\rangle=0
\right\},
\]
and
\[
\mathcal{N}=\left\{
u\in \mathcal{H}^{1,2}_0(\Omega),\;u\neq0,\;
\langle I'(u),u\rangle=0
\right\}.
\]
Moreover, the corresponding energy levels satisfy $m>2c.$}

\vskip2mm
\noindent
{\bf Remark 1.3.}  We would like to point out that our assumptions do not cover the asymptotically cubic nonlinearity
$\lambda u+\eta|u|^2u$ in \cite{Yu 2025} due to the presence of $(F_2)$; therefore, for this particular nonlinearity in \cite{Yu 2025}, the existence of the least energy sign-changing solution still remains open. Nevertheless, Theorem 1.1 provides, to the best of our knowledge, the first existence result on the least energy sign-changing solution for Kirchhoff equations with asymptotically cubic nonlinearity on locally finite graphs (even in the degenerate case
$b=0$).  Moreover,  in the continuous setting, when $N\geq 3$, the embedding $H_0^1(\Omega)\hookrightarrow L^p(\Omega)$ is compact only for $p<2^*:=\frac{2N}{N-2}$, so that the nonlinearity is usually required to have subcritical growth,  namely, \(\mathrm{(H_1)}\) holds. In contrast, since $\Omega$ is a bounded domain in a locally finite graph and $\mathcal{H}_0^{1,2}(\Omega)$ is finite-dimensional, continuously and compactly embedded into $L^q(\Omega)$ for every $1\leq q\leq+\infty$ so that in the present discrete graph setting, there is no finite Sobolev critical exponent imposing an upper bound on the polynomial growth of the nonlinearity, that is, $p\in[4,+\infty)$ in $(F_1)$.
\vskip2mm
\noindent
{\bf Remark 1.4.} There exist  examples satisfying our results.
\par
{\it Example 1. } For any $l>b\mu_0$ and $c>0$, let
\[
f_\infty(t)
=
-\frac{ct^3}{1+t^2},
\quad
f(t)
=
lt^3-\frac{ct^3}{1+t^2}.
\]
It is easy to see that $f(t)$ satisfies $(F_1)$ and $(F_2)$, while
$f_\infty(t)$ satisfies $(F_3)(I)$. Moreover,
\[
\left[
\frac{f_\infty(\tau)}{\tau^3}
-
\frac{f_\infty(t\tau)}{(t\tau)^3}
\right]
\operatorname{sign}(1-t)
=
\frac{c\tau^2|1-t^2|}
{(1+\tau^2)(1+t^2\tau^2)}
\ge 0.
\]
Thus, $(F_3)(II)$ holds for any $\theta_0\in(0,1)$. Hence, $f(t)$ satisfies conditions $(F_1)-(F_4)$.

\par
{\it Example 2.}
For any $l>b\mu_0$ and $\beta\in(0,a\lambda_1)$, let
\[
f_\infty(t)
=
\frac{\beta t^3}{1+t^4},
\quad
f(t)
=
lt^3+\frac{\beta t^3}{1+t^4}.
\]
Clearly, $f(t)$ satisfies $(F_1)$ and $(F_2)$, while $f_\infty(t)$ satisfies $(F_3)(I)$. Since $0<\beta<a\lambda_1$, there exists $\theta_0\in\left[\frac{\beta}{a\lambda_1},1\right)$.
For all $t>0$ and $\tau\in\mathbb{R}\setminus\{0\}$, noting that
\[
t^2\tau^6(1+t^2)
\le
(1+\tau^4)(1+t^4\tau^4),
\]
we have
\[
\begin{aligned}
&\left[
\frac{f_\infty(\tau)}{\tau^3}
-
\frac{f_\infty(t\tau)}{(t\tau)^3}
\right]
\operatorname{sign}(1-t)
+
\frac{a\theta_0\lambda_1|1-t^2|}{(t\tau)^2}
\\
&=
\frac{|1-t^2|}{t^2\tau^2}
\left[
a\theta_0\lambda_1
-
\frac{\beta t^2\tau^6(1+t^2)}
{(1+\tau^4)(1+t^4\tau^4)}
\right]
\\
&\ge
\frac{|1-t^2|}{t^2\tau^2}
\left(
a\theta_0\lambda_1-\beta
\right)
\ge 0.
\end{aligned}
\]
Thus, $(F_3)(II)$ holds, and hence $f(t)$ satisfies conditions $(F_1)-(F_4)$.

{\it Example 3.}
For any $l>b\mu_0$ and $\eta>0$, let
\[
f_\infty(t)
=
-\eta t^3e^{-t^2},
\quad
f(t)
=
lt^3-\eta t^3e^{-t^2}.
\]
It follows directly that $f(t)$ satisfies $(F_1)$ and $(F_2)$, while
$f_\infty(t)$ satisfies $(F_3)(I)$. Moreover, for all $t>0$ and
$\tau\in\mathbb{R}\setminus\{0\}$,
\[
\left[
\frac{f_\infty(\tau)}{\tau^3}
-
\frac{f_\infty(t\tau)}{(t\tau)^3}
\right]
\operatorname{sign}(1-t)
=
\eta
\left(
e^{-t^2\tau^2}-e^{-\tau^2}
\right)
\operatorname{sign}(1-t)
\ge 0.
\]
Thus, $(F_3)(II)$ holds for any $\theta_0\in(0,1)$.
Hence, we conclude that $f(t)$ satisfies conditions
$(F_1)-(F_4)$.

\vskip2mm
{\section{Preliminaries }}
\setcounter{equation}{0}
\noindent
\begin{lemma}\label{lem 6.6} (\cite{Grigor 2016}) {\it Let $G=(V,E)$ be a locally finite graph and $\Omega$ be a bounded domain satisfying $\Omega^\circ\not=\emptyset$. Then $\mathcal{H}^{1,2}_0(\Omega)\hookrightarrow L^q(\Omega)$ for all $1\leq q\leq+\infty$. Especially, if $1\leq q\leq +\infty$, then for all $\psi\in \mathcal{H}^{1,2}_0(\Omega)$,
\begin{eqnarray}
\label{b2}
\|\psi\|_{L^q(\Omega)}\leq K_q\|\psi\|,
\end{eqnarray}
where
$$
K_q=\frac{\left(\sum_{x\in V}\mu(x)\right)^{\frac{1}{q}}}{\mu_{\min}^{\frac{1}{2}}}, \ \mu_{\min}=\min_{x\in\Omega}\mu(x).
$$
In addition, $\mathcal{H}^{1,2}_0(\Omega)$ is pre-compact, namely, if $\{\psi_k\}$ is bounded in $\mathcal{H}^{1,2}_0(\Omega)$, then up to a subsequence, there exists some $\psi\in \mathcal{H}^{1,2}_0(\Omega)$ such that $\psi_k\rightarrow \psi$ in $\mathcal{H}^{1,2}_0(\Omega)$.}
\end{lemma}
\par

\vskip2mm
\begin{definition}
{\it A function \(u\) is a weak solution to equation \((\ref{eq1})\) provided that it satisfies the following identity:
\begin{equation}
\label{b5}
a\int_{\Omega\cup\partial\Omega}\nabla u\cdot\nabla v\,d\mu
+b\int_{\Omega\cup\partial\Omega}|\nabla u|^2\,d\mu
\int_{\Omega\cup\partial\Omega}\nabla u\cdot\nabla v\,d\mu
=
\int_{\Omega}f(u)v\,d\mu
\end{equation}
for any \(v\in\mathcal{H}^{1,2}_0(\Omega)\).}
\end{definition}

It follows from the definition of weak solutions that the critical points of \(I\) in \(\mathcal{H}_0^{1,2}(\Omega)\) are precisely the weak solutions of (\ref{eq1}).
\vskip2mm
\noindent
\begin{lemma}(\cite{Ou 2024})
{\it If \(u\in \mathcal{H}^{1,2}_0(\Omega)\) is a weak solution of equation \((\ref{eq1})\), then \(u\) is a point-wise solution of \((\ref{eq1})\).}
\end{lemma}

\vskip2mm
{\section{Proofs}}
\setcounter{equation}{0}
In this section, we shall complete the proofs of Theorem 1.1 by  a series of lemmas.
\par
Next, for any $u\in\mathcal{H}_0^{1,2}(\Omega)$, let
\begin{eqnarray*}
\Omega^+:=\{x\in\Omega:u(x)>0\},\;\;\Omega^-:=\{x\in\Omega:u(x)<0\}\;\;\text{and}\;\;
\Omega^\circ:=\{x\in\Omega\cup\partial\Omega:u(x)=0\}
\end{eqnarray*}
and we denote
\begin{eqnarray*}
\mathcal{W}_u:=
-\sum_{x\in\Omega^{-}}\sum_{\substack{y\sim x\\ y\in\Omega^{+}}}
w_{xy}u^{-}(x)u^{+}(y)
-\sum_{x\in\Omega^{+}}\sum_{\substack{y\sim x\\ y\in\Omega^{-}}}
w_{xy}u^{-}(y)u^{+}(x).
\end{eqnarray*}
It is easy to see that  $\mathcal{W}_u>0$ for any $u\in\mathcal{H}_0^{1,2}(\Omega)$.
\begin{lemma}\label{lem3.2}
{\it Suppose that \((F_1)-(F_4)\) hold. For all \(u\in \mathcal{H}_0^{1,2}(\Omega)\) and \(s,t\geq 0\), it holds that}
\begin{align*}
I(u) \geq\;& I(su^{+}+tu^{-})
+\frac{1-s^{4}}{4}\left\langle I'(u),u^{+}\right\rangle
+\frac{1-t^{4}}{4}\left\langle I'(u),u^{-}\right\rangle  \\
&+\frac{a(1-\theta_{0})}{4}(1-s^{2})^{2}\|u^{+}\|^{2}
+\frac{a(1-\theta_{0})}{4}(1-t^{2})^{2}\|u^{-}\|^{2}
+\frac{a(s-t)^{2}}{4}\mathcal{W}_u
+\frac{b(s^{2}-t^{2})^{2}}{4}\|u^{+}\|^{2}\|u^{-}\|^{2}  \\
&+\frac{b}{8}(s-t)^{2}(3s^{2}+2st+t^{2})\|u^{+}\|^{2}
\mathcal{W}_u
+\frac{b}{8}(s-t)^{2}(3t^{2}+2st+s^{2})\|u^{-}\|^{2}
\mathcal{W}_u
+\frac{b(s^{2}-t^{2})^{2}}{8}
\mathcal{W}_u^2.
\end{align*}
\end{lemma}

\begin{proof}[\bf{Proof}]
It follows from $(F_3)$ that
\begin{equation}\label{a2}
\left[
\frac{f(\tau)}{\tau^{3}}
-
\frac{f(t\tau)}{(t\tau)^{3}}
\right]\operatorname{sign}(1-t)
+
\frac{a\theta_{0}\lambda_{1}|1-t^{2}|}{(t\tau)^{2}}
\geq 0.
\end{equation}
Thus, \eqref{a2} implies that
\begin{align}\label{a1}
&\frac{1-t^{4}}{4}f(\tau)\tau+F(t\tau)-F(\tau)
+\frac{a\theta_{0}\lambda_{1}}{4}(1-t^{2})^{2}\tau^{2}
\notag\\
  =&\int_{t}^{1}\left[
\frac{f(\tau)}{\tau^{3}}
-\frac{f(s\tau)}{(s\tau)^{3}}
+\frac{a\theta_{0}\lambda_{1}(1-s^{2})}{(s\tau)^{2}}
\right]s^{3}\tau^{4}\,ds
\geq 0, \quad \forall\, t\geq 0,\ \tau\in\mathbb{R}\setminus\{0\}.
\end{align}
Then, by the Poincar\'e inequality
$\lambda_{1}\int_{\Omega}|u^{\pm}|^{2}d\mu \leq \|u^{\pm}\|^{2}$ and \eqref{a1}, we have
\begin{equation}\label{a5}
\int_{\Omega}F(u^{+})d\mu
-
\int_{\Omega}F(su^{+})d\mu
\leq
\frac{1-s^{4}}{4}\int_{\Omega}f(u^{+})u^{+}d\mu
+
\frac{a\theta_{0}}{4}(1-s^{2})^{2}\|u^{+}\|^{2}.
\end{equation}
and
\begin{equation}\label{a6}
\int_{\Omega}F(u^{-})d\mu
-
\int_{\Omega}F(tu^{-})d\mu
\leq
\frac{1-t^{4}}{4}\int_{\Omega}f(u^{-})u^{-}d\mu
+
\frac{a\theta_{0}}{4}(1-t^{2})^{2}\|u^{-}\|^{2}.
\end{equation}
Combining \eqref{a5} and \eqref{a6}, we obtain
\begin{align}\label{a7}
\int_{\Omega}F(u)\,d\mu
-
\int_{\Omega}F(su^{+}+tu^{-})\,d\mu
\notag
\leq&
\frac{1-s^{4}}{4}\int_{\Omega}f(u^{+})u^{+}\,d\mu
+
\frac{1-t^{4}}{4}\int_{\Omega}f(u^{-})u^{-}\,d\mu
\notag\\
&+
\frac{a\theta_{0}}{4}(1-s^{2})^{2}\|u^{+}\|^{2}
+
\frac{a\theta_{0}}{4}(1-t^{2})^{2}\|u^{-}\|^{2}.
\end{align}
Moreover, by Proposition 3.2 in \cite{Ou 2024} and (\ref{a7}), we have
\begin{align*}
& I(u)-I(su^{+}+tu^{-})
-\frac{1-s^{4}}{4}\langle I'(u),u^{+}\rangle
-\frac{1-t^{4}}{4}\langle I'(u),u^{-}\rangle
\\
={}&
\frac{a}{2}
(\|u^{+}\|^{2}+\|u^{-}\|^{2}+\mathcal W_u)
-\frac{a}{2}
(s^{2}\|u^{+}\|^{2}+t^{2}\|u^{-}\|^{2}+st\mathcal W_u)
+\frac{b}{4}
(\|u^{+}\|^{2}+\|u^{-}\|^{2}+\mathcal W_u)^{2}
\\
&-\frac{b}{4}
(s^{2}\|u^{+}\|^{2}+t^{2}\|u^{-}\|^{2}+st\mathcal W_u)^{2}
-\int_{\Omega}F(u^{+})
-\int_{\Omega}F(u^{-})
+\int_{\Omega}F(su^{+})
+\int_{\Omega}F(tu^{-})
\\
&-\frac{1-s^{4}}{4}
\left[
a\left(\|u^{+}\|^{2}+\frac12\mathcal W_u\right)
+b\left(\|u^{+}\|^{2}+\|u^{-}\|^{2}+\mathcal W_u\right)
\left(\|u^{+}\|^{2}+\frac12\mathcal W_u\right)
-\int_{\Omega}f(u^{+})u^{+}
\right]
\\
&-\frac{1-t^{4}}{4}
\left[
a\left(\|u^{-}\|^{2}+\frac12\mathcal W_u\right)
+b\left(\|u^{+}\|^{2}+\|u^{-}\|^{2}+\mathcal W_u\right)
\left(\|u^{-}\|^{2}+\frac12\mathcal W_u\right)
-\int_{\Omega}f(u^{-})u^{-}
\right].
\\
\geq{}&
\frac{a(1-s^{2})^{2}}{4}\|u^{+}\|^{2}
+\frac{a(1-t^{2})^{2}}{4}\|u^{-}\|^{2}
+\frac{a}{8}(s^{4}+t^{4}+2-4st)\mathcal W_u
+\frac{b(s^{2}-t^{2})^{2}}{4}
\|u^{+}\|^{2}\|u^{-}\|^{2}
\\
&+
\frac{b}{8}(s-t)^{2}
(3s^{2}+2st+t^{2})
\|u^{+}\|^{2}\mathcal W_u
+
\frac{b}{8}(s-t)^{2}
(3t^{2}+2st+s^{2})
\|u^{-}\|^{2}\mathcal W_u
\\
&+
\frac{b(s^{2}-t^{2})^{2}}{8}\mathcal W_u^{2}
-
\frac{a\theta_{0}}{4}(1-s^{2})^{2}\|u^{+}\|^{2}
-
\frac{a\theta_{0}}{4}(1-t^{2})^{2}\|u^{-}\|^{2}.
\\
={}&
\frac{a(1-\theta_{0})}{4}(1-s^{2})^{2}\|u^{+}\|^{2}
+\frac{a(1-\theta_{0})}{4}(1-t^{2})^{2}\|u^{-}\|^{2}
+\frac{a}{8}(s^{4}+t^{4}+2-4st)\mathcal W_u
+\frac{b(s^{2}-t^{2})^{2}}{4}\|u^{+}\|^{2}\|u^{-}\|^{2}
\\
&+
\frac{b}{8}(s-t)^{2}
(3s^{2}+2st+t^{2})\|u^{+}\|^{2}\mathcal W_u
+\frac{b}{8}(s-t)^{2}
(3t^{2}+2st+s^{2})\|u^{-}\|^{2}\mathcal W_u
+\frac{b(s^{2}-t^{2})^{2}}{8}\mathcal W_u^{2}.
\end{align*}
Furthermore, by $s^{4}+t^{4}+2-4st \geq 2(s-t)^{2}$, and the nonnegativity of $\mathcal{W}_u$, we obtain the conclusion.
\end{proof}
\par
\noindent{\bf Remark 3.1.}\label{rem3.1}
Let \(s=t\) in Lemma \ref{lem3.2}. For all
\(u\in \mathcal{H}_0^{1,2}(\Omega)\) and \(t\geq 0\), it holds that
\[
I(u)\geq I(tu)
+\frac{1-t^4}{4}\left\langle I'(u),u\right\rangle
+\frac{a(1-\theta_0)}{4}(1-t^2)^2
\left(\|u^+\|^2+\|u^-\|^2\right).
\]

\vskip2mm
\par
We can also obtain the following result by using the similar proof as that in Lemma \ref{lem3.2}.
\par

\begin{lemma}\label{lem3.3}
{\it Suppose that \((F_1)-(F_4)\) hold. For all
\(u\in \mathcal{H}_0^{1,2}(\Omega)\) and \(t\geq 0\), it holds that
\[
I(u)\geq I(tu)
+\frac{1-t^4}{4}\left\langle I'(u),u\right\rangle
+\frac{a(1-\theta_0)(1-t^2)^2}{4}\|u\|^2.
\]}
\end{lemma}

\begin{proof}[\bf{Proof}]
By (\ref{b3}), (\ref{b4}) and (\ref{a1}), we have
\begin{align*}
I(u)-I(tu)
={}&
\frac{a}{2}(1-t^2)\|u\|^2
+\frac{b}{4}(1-t^4)\|u\|^4
-\int_{\Omega}\left[F(u)-F(tu)\right]\,d\mu  \\
={}&
\frac{a(1-t^4)}{4}\|u\|^2
+\frac{b(1-t^4)}{4}\|u\|^4
-\frac{1-t^4}{4}\int_{\Omega}f(u)u\,d\mu
+\frac{a(1-t^2)^2}{4}\|u\|^2  \\
&+
\frac{1-t^4}{4}\int_{\Omega}f(u)u\,d\mu
-\int_{\Omega}\left[F(u)-F(tu)\right]\,d\mu  \\
={}&
\frac{1-t^4}{4}\left\langle I'(u),u\right\rangle
+\frac{a(1-t^2)^2}{4}\|u\|^2
+\int_{\Omega}
\left[
\frac{1-t^4}{4}f(u)u
-F(u)+F(tu)
\right]\,d\mu  \\
\geq{}&
\frac{1-t^4}{4}\left\langle I'(u),u\right\rangle
+\frac{a(1-t^2)^2}{4}\|u\|^2
-\frac{a\theta_0\lambda_1(1-t^2)^2}{4}
\int_{\Omega}|u|^2\,d\mu  \\
\geq{}&
\frac{1-t^4}{4}\left\langle I'(u),u\right\rangle
+\frac{a(1-t^2)^2}{4}\|u\|^2
-\frac{a\theta_0(1-t^2)^2}{4}\|u\|^2  \\
={}&
\frac{1-t^4}{4}\left\langle I'(u),u\right\rangle
+\frac{a(1-\theta_0)(1-t^2)^2}{4}\|u\|^2 .
\end{align*}
Thus, the proof is complete.
\end{proof}
From the definitions of \(\mathcal{M}\) and \(\mathcal{N}\), together with Lemma \ref{lem3.2} and \ref{lem3.3}, we directly obtain the following two corollaries.
\begin{corollary}\label{cor3.4}
{\it Suppose that \((F_1)-(F_4)\) hold. If \(u=u^{+}+u^{-}\in \mathcal{M}\), then}
\begin{align*}
I(u) \geq\;& I(su^{+}+tu^{-})+\frac{a(1-\theta_{0})}{4}(1-s^{2})^{2}\|u^{+}\|^{2}
+\frac{a(1-\theta_{0})}{4}(1-t^{2})^{2}\|u^{-}\|^{2}  \\
&+\frac{a(s-t)^{2}}{4}
\mathcal{W}_u  +\frac{b(s^{2}-t^{2})^{2}}{4}\|u^{+}\|^{2}\|u^{-}\|^{2}  \\
&+\frac{b}{8}(s-t)^{2}(3s^{2}+2st+t^{2})\|u^{+}\|^{2}
\mathcal{W}_u
+\frac{b}{8}(s-t)^{2}(3t^{2}+2st+s^{2})\|u^{-}\|^{2}
\mathcal{W}_u
+\frac{b(s^{2}-t^{2})^{2}}{8}
\mathcal{W}_u^{2}, \quad \forall\,s,t\geq 0 .
\end{align*}
\end{corollary}

\begin{corollary}\label{cor3.5}
{\it Suppose that \((F_1)-(F_4)\) hold. If \(u\in\mathcal N\), then}
\[
I(u)\geq I(tu)
+\frac{a(1-\theta_0)}{4}(1-t^2)^2\|u\|^2,
\quad \forall t\geq0 .
\]
\end{corollary}

\vskip2mm
\par
Furthermore, by the nonnegativity of $\mathcal{W}_u$, Corollary \ref{cor3.4} and Corollary \ref{cor3.5}, it is easy to obtain the following conclusions.
\begin{corollary}\label{cor3.6}
{\it Suppose that \((F_1)-(F_4)\) hold. If \(u=u^{+}+u^{-}\in \mathcal{M}\), then
$I(u)=\max\limits_{s,t\geq 0} I(su^+ + tu^-)$.}
\end{corollary}

\begin{corollary}\label{cor3.7}
{\it Suppose that \((F_1)-(F_4)\) hold. If \(u\in\mathcal N\), then $I(u)=\max\limits_{t\geq0} I(tu).$}

\end{corollary}

\begin{lemma}\label{lem3.8}
{\it Suppose that \((F_3)\) holds. Then}
\[
f_{\infty}(\tau)\tau \leq a\theta_0\lambda_1 \tau^2,
\quad \forall\,\tau\in\mathbb{R}.
\]
\end{lemma}

\begin{proof}[\bf{Proof}]
For all \(t>0\) and \(\tau\in\mathbb{R}\setminus\{0\}\), it follows from \((II)\) in \((F_3)\) that
\[
\left[
\frac{f_{\infty}(\tau)}{\tau^3}
-
\frac{f_{\infty}(t\tau)}{(t\tau)^3}
\right]\operatorname{sign}(1-t)
+
\frac{a\theta_0\lambda_1|1-t^2|}{(t\tau)^2}
\geq 0 .
\]
Taking the limit as $t \rightarrow +\infty$ in the above inequality, we deduce from $(I)$ in $(F_3)$ that

\[-
\frac{f_{\infty}(\tau)}{\tau^3}
+
\frac{a\theta_0\lambda_1}{\tau^2}
\geq 0,\quad \forall\,\tau\in\mathbb{R}.
\]
Hence,
\[
f_{\infty}(\tau)\tau \leq a\theta_0\lambda_1 \tau^2,
\quad \forall\,\tau\in\mathbb{R}.
\]
The proof is complete.
\end{proof}

 Together with the Poincar\'e inequality and Lemma \ref{lem3.8}, we can obtain that
\begin{equation}\label{c3}
\int_{\Omega}f_{\infty}(u^\pm)u^\pm\,d\mu
\leq
a\theta_0\|u^\pm\|^2.
\end{equation}

Define a set \(E\) as follows:
\begin{equation}\label{c1}
E=\left\{
u\in \mathcal{H}_0^{1,2}(\Omega):\
u^\pm\neq 0,\
b\|u\|^2\!\int_{\Omega\cup\partial\Omega}\!\Gamma(u,u^\pm)\,d\mu
<l\!\int_{\Omega}\!|u^\pm|^4\,d\mu
\right\}.
\end{equation}
We claim that  \(E\) is not empty under the assumptions \((F_1)-(F_4)\). Indeed, since \(l>b\mu_0\), by the definition of \(\mu_0\), there exists
\(u\in H_0^{1,2}(\Omega)\) with \(u^\pm\neq0\) such that
\[
b\|u\|^2\int_{\Omega}\Gamma(u,u^\pm)d\mu
<l\int_{\Omega}|u^\pm|^4d\mu .
\]
Combining \(u=u^++u^-\) with the bilinearity of \(\Gamma\), we obtain
\[
b\|u\|^4
<l\int_{\Omega}(|u^+|^4+|u^-|^4)d\mu
=l\int_{\Omega}|u|^4d\mu .
\]
Thus, \(u\in E\), which implies that \(E\neq\emptyset\).
\begin{lemma}\label{lem3.9}
{\it Suppose that \((F_1)-(F_4)\) hold. If \(u\in \mathcal{H}_0^{1,2}(\Omega)\) with \(u^\pm\neq 0\) and
$\langle I'(u),u^\pm\rangle\leq 0$,
then \(u\in E\).}
\end{lemma}

\begin{proof}[\bf{Proof}]
For any \(u\in \mathcal{H}_0^{1,2}(\Omega)\) satisfying \(u^\pm\neq 0\) and
$\langle I'(u),u^\pm\rangle\leq 0$, we have
\[
a\int_{\Omega\cup\partial\Omega}\Gamma(u,u^\pm)\,d\mu
+
b\|u\|^2
\int_{\Omega\cup\partial\Omega}\Gamma(u,u^\pm)\,d\mu
\leq
\int_{\Omega}f(u)u^\pm\,d\mu .
\]
Then it follows from  $(F_3)$ that
\[
\int_{\Omega}f(u^\pm)u^\pm\,d\mu
=
l\int_{\Omega}|u^\pm|^4\,d\mu
+
\int_{\Omega}f_\infty(u^\pm)u^\pm\,d\mu .
\]
and
\begin{equation}\label{c4}
b\|u\|^2
\int_{\Omega\cup\partial\Omega}\Gamma(u,u^\pm)\,d\mu
-
l\int_{\Omega}|u^\pm|^4\,d\mu
\leq
\int_{\Omega}f_\infty(u^\pm)u^\pm\,d\mu
-
a\int_{\Omega\cup\partial\Omega}\Gamma(u,u^\pm)\,d\mu.
\end{equation}
Moreover, by inequalities (20) and (21) in \cite{Ou 2024}, we have
\begin{equation}\label{c5}
\int_{\Omega\cup\partial\Omega}\Gamma(u,u^\pm)\,d\mu
\geq
\|u^\pm\|^2.
\end{equation}
Hence, it follows from (\ref{c3}), (\ref{c4}) and (\ref{c5}) that
\[
\begin{aligned}
& b\|u\|^2
\int_{\Omega\cup\partial\Omega}\Gamma(u,u^\pm)\,d\mu
-
l\int_{\Omega}|u^\pm|^4\,d\mu  \\
&\leq
a\theta_0\|u^\pm\|^2
-
a\int_{\Omega\cup\partial\Omega}\Gamma(u,u^\pm)\,d\mu  \\
&\leq
a(\theta_0-1)\|u^\pm\|^2,
\end{aligned}
\]
which implies that \(u\in E\) due to \(0<\theta_0<1\). The proof is complete.
\end{proof}
By the definition of \(\mathcal{M}\) and Lemma \ref{lem3.9}, we directly conclude the following corollary.

\begin{corollary}\label{cor3.10}
{\it Suppose that \((F_1)-(F_4)\) hold. Then}
$
\mathcal{M}\subset E .
$
\end{corollary}

\begin{lemma}\label{lem3.11}
{\it Suppose that \((F_1)-(F_4)\) hold. For any \(u\in E\), there exists a unique pair of positive numbers $(s_u,t_u)$ such that
$
s_u u^+ + t_u u^- \in \mathcal{M}.
$}
\end{lemma}

\begin{proof}[\bf{Proof}]
First, we prove the existence of a pair $(s_u,t_u)$ for any $u\in \mathcal{H}^{1,2}_0(\Omega)$ with $u^\pm\neq 0$. By Proposition 3.2 in \cite{Ou 2024}, we know that
\begin{eqnarray}\label{a10}
G(s,t)
&:=&
\left\langle I'(su^++tu^-),su^+\right\rangle \nonumber\\
&=&
as^2\|u^+\|^2
+bs^4\|u^+\|^4
-\int_\Omega f(su^+)su^+\,d\mu
+\frac{ast}{2}\mathcal W_u
+\frac{3bs^3t}{2}\|u^+\|^2\mathcal W_u
+\frac{bst^3}{2}\|u^-\|^2\mathcal W_u
\nonumber\\
&&
+\frac{bs^2t^2}{2}
\left[
\sum_{x\in\Omega^-}
\sum_{\substack{y\sim x\\y\in\Omega^+}}
w_{xy}u^-(x)u^+(y)
\right]^2
+\frac{bs^2t^2}{2}
\left[
\sum_{x\in\Omega^+}
\sum_{\substack{y\sim x\\y\in\Omega^-}}
w_{xy}u^-(y)u^+(x)
\right]^2
\nonumber\\
&&
+bs^2t^2\|u^+\|^2\|u^-\|^2
+bs^2t^2
\sum_{x\in\Omega^+}
\sum_{\substack{y\sim x\\y\in\Omega^-}}
w_{xy}u^-(y)u^+(x)
\sum_{x\in\Omega^-}
\sum_{\substack{y\sim x\\y\in\Omega^+}}
w_{xy}u^-(x)u^+(y).
\end{eqnarray}
and
\begin{eqnarray}\label{a11}
H(s,t)
&:=&
\left\langle I'(su^++tu^-),tu^-\right\rangle \nonumber\\
&=&
at^2\|u^-\|^2
+bt^4\|u^-\|^4
-\int_\Omega f(tu^-)tu^-\,d\mu
+\frac{ast}{2}\mathcal W_u
+\frac{3bst^3}{2}\|u^-\|^2\mathcal W_u
+\frac{bs^3t}{2}\|u^+\|^2\mathcal W_u
\nonumber\\
&&
+\frac{bs^2t^2}{2}
\left[
\sum_{x\in\Omega^-}
\sum_{\substack{y\sim x\\y\in\Omega^+}}
w_{xy}u^-(x)u^+(y)
\right]^2
+\frac{bs^2t^2}{2}
\left[
\sum_{x\in\Omega^+}
\sum_{\substack{y\sim x\\y\in\Omega^-}}
w_{xy}u^-(y)u^+(x)
\right]^2
\nonumber\\
&&
+bs^2t^2\|u^+\|^2\|u^-\|^2
+bs^2t^2
\sum_{x\in\Omega^+}
\sum_{\substack{y\sim x\\y\in\Omega^-}}
w_{xy}u^-(y)u^+(x)
\sum_{x\in\Omega^-}
\sum_{\substack{y\sim x\\y\in\Omega^+}}
w_{xy}u^-(x)u^+(y).
\end{eqnarray}
By \((F_1)\) and \((F_2)\), for any \(\xi>0\), there exists
\(C_\xi>0\) such that
\begin{equation}\label{a12}
|f(\tau)\tau|
\leq
\xi|\tau|^2+C_\xi|\tau|^p,
\quad \forall\,\tau\in\mathbb{R}.
\end{equation}
Note that $u^-(y)u^+(x)<0$ and $u^+(y)u^-(x)<0$. Then combining (\ref{a10}), (\ref{a12}) and Lemma \ref{lem 6.6}, there exist positive constants
$0<\xi_1<\frac{a}{K_2^2}$and \(C_{\xi_1}>0\) such that
\begin{align}\label{a13}
G(s,s)
\geq{}&
as^2\|u^+\|^2
+bs^4\|u^+\|^4
-\int_\Omega f(su^+)su^+\,d\mu
\notag\\
\geq{}&
as^2\|u^+\|^2
+bs^4\|u^+\|^4
-\xi_1s^2\int_\Omega|u^+|^2\,d\mu
-C_{\xi_1}s^p\int_\Omega|u^+|^p\,d\mu
\notag\\
={}&
as^2\|u^+\|^2
+bs^4\|u^+\|^4
-\xi_1s^2\|u^+\|_{L^2(\Omega)}^2
-C_{\xi_1}s^p\|u^+\|_{L^p(\Omega)}^p
\notag\\
\geq{}&
as^2\|u^+\|^2
+bs^4\|u^+\|^4
-\xi_1s^2K_2^2\|u^+\|^2
-C_{\xi_1}s^pK_p^p\|u^+\|^p
\notag\\
={}&
\left(a-\xi_1K_2^2\right)s^2\|u^+\|^2
+bs^4\|u^+\|^4
-C_{\xi_1}s^pK_p^p\|u^+\|^p.
\end{align}
Then, it follows from \(a-\xi_1K_2^2>0\) and \(p\geq4\) that $G(s,s)>0$ for all sufficiently small \(s>0\). Similar to the argument used in (\ref{a13}), we can also obtain that $H(s,s)>0$ for \(s>0\) sufficiently small. Moreover, it follows from (\ref{a10}) that
\begin{eqnarray}\label{a14}
G(t,t)
& = &at^2\|u^+\|^2+bt^4\|u^+\|^4
-lt^4\int_\Omega|u^+|^4\,d\mu
-\int_\Omega f_\infty(tu^+)tu^+\,d\mu
+\frac{at^2}{2}\mathcal{W}_u
\nonumber\\
&   &+\frac{3bt^4}{2}\|u^+\|^2\mathcal{W}_u
+\frac{bt^4}{2}\|u^-\|^2\mathcal{W}_u
+bt^4\|u^+\|^2\|u^-\|^2
\nonumber\\
&   &+\frac{bt^4}{2}
\left[
\sum_{x\in\Omega^-}
\sum\limits_{\substack{y\thicksim x\\y\in\Omega^+}}
w_{xy}u^-(x)u^+(y)
\right]^2
+\frac{bt^4}{2}
\left[
\sum_{x\in\Omega^+}
\sum\limits_{\substack{y\thicksim x\\y\in\Omega^-}}
w_{xy}u^-(y)u^+(x)
\right]^2
\nonumber\\
&   &+bt^4
\sum_{x\in\Omega^+}
\sum\limits_{\substack{y\thicksim x\\y\in\Omega^-}}
w_{xy}u^-(y)u^+(x)
\sum_{x\in\Omega^-}
\sum\limits_{\substack{y\thicksim x\\y\in\Omega^+}}
w_{xy}u^-(x)u^+(y).
\end{eqnarray}
Then, by (\ref{c1}), (\ref{a14}) and condition $(F_3)$, we obtain that
\begin{align}\label{a15}
\limsup_{|t|\to\infty}\frac{G(t,t)}{t^4}
={}&
b\|u^+\|^4
-l\int_\Omega |u^+|^4\,d\mu
+b\|u^+\|^2\|u^-\|^2
+\frac{3b}{2}\|u^+\|^2\mathcal W_u
+\frac{b}{2}\|u^-\|^2\mathcal W_u
\notag\\
&+
\frac{b}{2}
\left[
\sum_{x\in\Omega^-}
\sum\limits_{\substack{y\thicksim x\\y\in\Omega^+}}
w_{xy}u^-(x)u^+(y)
\right]^2
+
\frac{b}{2}
\left[
\sum_{x\in\Omega^+}
\sum\limits_{\substack{y\thicksim x\\y\in\Omega^-}}
w_{xy}u^-(y)u^+(x)
\right]^2
\notag\\
&+
b
\sum_{x\in\Omega^+}
\sum\limits_{\substack{y\thicksim x\\y\in\Omega^-}}
w_{xy}u^-(y)u^+(x)
\sum_{x\in\Omega^-}
\sum\limits_{\substack{y\thicksim x\\y\in\Omega^+}}
w_{xy}u^-(x)u^+(y)
\notag\\
={}&
b\|u\|^2
\int_{\Omega\cup\partial\Omega}
\Gamma(u,u^+)\,d\mu
-l\int_\Omega |u^+|^4\,d\mu
<0 .
\end{align}
Thus, (\ref{a15}) implies that $G(t,t)<0$ for \(t>0\) large enough. By a similar argument to that used in (\ref{a15}), we also obtain that $H(t,t)<0$ for \(t>0\) large enough. Consequently, there exist constants $0<\alpha<\beta$ such that
\begin{equation}\label{a16}
G(\alpha,\alpha)>0,\qquad
H(\alpha,\alpha)>0,\qquad
G(\beta,\beta)<0,\qquad
H(\beta,\beta)<0.
\end{equation}
For every fixed \(s>0\), it is easy to see that the function \(t\mapsto G(s,t)\) is nondecreasing on \([0,+\infty)\). Similarly, for every fixed \(t>0\), the function \(s\mapsto H(s,t)\) is also nondecreasing on \([0,+\infty)\). Then, for all $s,t\in[\alpha,\beta]$, it follows from
(\ref{a10}), (\ref{a11}), and (\ref{a16}) that
\begin{align}\label{a17}
G(\alpha,t)
\geq{}G(\alpha,\alpha)>0,\qquad
G(\beta,t)\leq G(\beta,\beta)<0,\qquad
H(s,\alpha)\geq H(\alpha,\alpha)>0,\qquad
H(s,\beta)\leq H(\beta,\beta)<0.
\end{align}
In view of the continuity of \(f\) and the finiteness of the graph domain \(\Omega\), both \(G\) and \(H\) are continuous on \([\alpha,\beta]\times[\alpha,\beta]\). Therefore, the mapping
$
(G,H):[\alpha,\beta]\times[\alpha,\beta]
\longrightarrow
\mathbb{R}\times\mathbb{R}
$
is continuous. Thus, combining (\ref{a17}) with the Poincar\'{e}-Miranda theorem \cite{Kulpa 1997}, we obtain a point $(s_u,t_u)\in(\alpha,\beta)\times(\alpha,\beta)$ satisfying $G(s_u,t_u)=H(s_u,t_u)=0,$
which implies that $(s_u,t_u)$ is a pair of positive numbers such that $s_uu^++t_uu^-\in\mathcal{M}$.
\par
Next, we prove the uniqueness of $(s_u,t_u)$. Suppose, on the contrary,
that there exist two distinct pairs of positive numbers $(s_1,t_1)$ and
$(s_2,t_2)$ such that $(s_1,t_1)\neq(s_2,t_2)$,
$s_1u^++t_1u^-\in\mathcal{M}$, and
$s_2u^++t_2u^-\in\mathcal{M}$. By Lemma \ref{lem3.2}, taking
$u :=s_1u^++t_1u^-$, $s :=\frac{s_2}{s_1}$, and
$t :=\frac{t_2}{t_1}$, we have
\begin{align}\label{a19}
&I(s_1u^+ + t_1u^-)
\notag\\
\geq{}&
I(s_2u^+ + t_2u^-)
+\frac{a(1-\theta_0)(s_1^2-s_2^2)^2}{4s_1^2}\|u^+\|^2
+\frac{a(1-\theta_0)(t_1^2-t_2^2)^2}{4t_1^2}\|u^-\|^2
\notag\\
&+
\frac{a(s_2t_1-s_1t_2)^2}{4s_1t_1}\mathcal W_u
+
\frac{b(s_2^2t_1^2-s_1^2t_2^2)^2}{4s_1^2t_1^2}
\|u^+\|^2\|u^-\|^2
+
\frac{b(s_2^2t_1^2-s_1^2t_2^2)^2}{8s_1^2t_1^2}
\mathcal W_u^2
\notag\\
&+
\frac{b(s_2t_1-s_1t_2)^2
(3s_2^2t_1^2+2s_1s_2t_1t_2+s_1^2t_2^2)}
{8s_1t_1^3}
\|u^+\|^2\mathcal W_u
\notag\\
&+
\frac{b(s_2t_1-s_1t_2)^2
(3s_1^2t_2^2+2s_1s_2t_1t_2+s_2^2t_1^2)}
{8s_1^3t_1}
\|u^-\|^2\mathcal W_u
\notag\\
>&I(s_2u^+ + t_2u^-).
\end{align}
Similarly, by taking $u:=s_2u^++t_2u^-$, $s:=\frac{s_1}{s_2}$, and $t:=\frac{t_1}{t_2}$, we obtain that
\begin{align}\label{a20}
&I(s_2u^+ + t_2u^-)
\notag\\
\geq{}&
I(s_1u^+ + t_1u^-)
+\frac{a(1-\theta_0)(s_2^2-s_1^2)^2}{4s_2^2}\|u^+\|^2
+\frac{a(1-\theta_0)(t_2^2-t_1^2)^2}{4t_2^2}\|u^-\|^2
\notag\\
&+
\frac{a(s_1t_2-s_2t_1)^2}{4s_2t_2}\mathcal W_u
+
\frac{b(s_1^2t_2^2-s_2^2t_1^2)^2}
{4s_2^2t_2^2}
\|u^+\|^2\|u^-\|^2
+
\frac{b(s_1^2t_2^2-s_2^2t_1^2)^2}
{8s_2^2t_2^2}
\mathcal W_u^2
\notag\\
&+
\frac{b(s_1t_2-s_2t_1)^2
(3s_1^2t_2^2+2s_1s_2t_1t_2+s_2^2t_1^2)}
{8s_2t_2^3}
\|u^+\|^2\mathcal W_u
\notag\\
&+
\frac{b(s_1t_2-s_2t_1)^2
(3s_2^2t_1^2+2s_1s_2t_1t_2+s_1^2t_2^2)}
{8s_2^3t_2}
\|u^-\|^2\mathcal W_u
\notag\\
>{}&
I(s_1u^+ + t_1u^-).
\end{align}
Both (\ref{a19}) and (\ref{a20}) imply $(s_1,t_1)=(s_2,t_2)$. That is, there exists a unique pair of positive numbers $(s_u,t_u)$ such that $s_uu^++t_uu^-\in\mathcal{M}$. The proof is complete.
\end{proof}

\begin{lemma}\label{lem3.12}
{\it Suppose that \((F_1)-(F_4)\) hold. Then
\[
\inf_{u\in\mathcal{M}} I(u)
=:m
=
\inf_{u\in E}\max_{s,t\geq0} I(su^+ + tu^-).
\]}
\end{lemma}

\begin{proof}[\bf{Proof}]
On the one hand, it follows from Corollary \ref{cor3.6} and Corollary \ref{cor3.10} that
\begin{equation}\label{a21}
\inf_{u\in E}\max_{s,t\geq0} I(su^+ + tu^-)
\leq
\inf_{u\in\mathcal{M}}\max_{s,t\geq0} I(su^+ + tu^-)
=
\inf_{u\in\mathcal{M}} I(u)
=
m.
\end{equation}
On the other hand, for any \(u\in E\), Lemma \ref{lem3.11} yields
\[
\max_{s,t\geq0} I(su^+ + tu^-)
\geq
I(s_u u^+ + t_u u^-)
\geq
\inf_{u\in\mathcal{M}} I(u)
=
m.
\]
Therefore, we deduce that
\begin{equation}\label{a22}
\inf_{u\in E}\max_{s,t\geq0} I(su^+ + tu^-)
\geq
\inf_{u\in\mathcal{M}} I(u)
=
m.
\end{equation}
Combining (\ref{a21}) and (\ref{a22}), we conclude the desired result.
\end{proof}

\begin{lemma}\label{lem3.13}
{\it Suppose that \((F_3)\) holds. Then}
\[
\frac{1}{4}f(\tau)\tau-F(\tau)
+\frac{a\theta_{0}\lambda_{1}}{4}\tau^{2}\geq 0,
\quad \forall\,\tau\in\mathbb{R}.
\]
\end{lemma}
\begin{proof}[\bf{Proof}]
The result follows immediately from (\ref{a1}) by setting $t=0$.
\end{proof}
\vskip2mm
\noindent

Define a set \(G\) as follows:
\begin{equation}\label{c2}
G=
\left\{
u\in \mathcal{H}_0^{1,2}(\Omega)\setminus\{0\}:
b\|u\|^4<l\int_{\Omega}|u|^4\,d\mu
\right\}.
\end{equation}
Similar to the argument for the set $E$, it is easy to see that $l>b\mu_0$, which implies that $G$ is nonempty.
\begin{lemma}\label{lem3.14}
{\it Suppose that \((F_1)-(F_4)\) hold. If \(u\in \mathcal{H}_0^{1,2}(\Omega)\setminus\{0\}\) and \(\left\langle I'(u),u\right\rangle\leq0\), then \(u\in G\).}
\end{lemma}
\begin{proof}[\bf{Proof}]
For any \(u\in \mathcal{H}_0^{1,2}(\Omega)\) satisfying \(u\neq 0\) and
$\langle I'(u),u\rangle\leq 0$, we first obtain
\[
a\|u\|^2+b\|u\|^4
\leq
\int_{\Omega}f(u)u\,d\mu .
\]
By condition \((F_3)\) and Lemma \ref{lem3.8}, it follows that
\[
a\|u\|^2+b\|u\|^4
\leq
\int_{\Omega}f(u)u\,d\mu
\leq
l\int_{\Omega}|u|^4\,d\mu
+
a\theta_0\|u\|^2 .
\]
Hence,
\[
b\|u\|^4-l\int_{\Omega}|u|^4\,d\mu
\leq
a(\theta_0-1)\|u\|^2<0.
\]
This completes the proof.
\end{proof}
By the definition of \(\mathcal N\) and Lemma \ref{lem3.14}, we directly obtain the following corollary.
\begin{corollary}\label{cor3.15}
{\it Suppose that \((F_1)-(F_4)\) hold. Then \(\mathcal N\subset G\).}
\end{corollary}
\begin{lemma}\label{lem3.16}
{\it Suppose that \((F_1)-(F_4)\) hold. If \(u\in G\), there exists a unique \(t_u>0\) such that \(t_uu\in\mathcal N\).}
\end{lemma}

\begin{proof}[\bf{Proof}]
First, we establish the existence of \(t_u\). For any \(u\in G\), define the function
\(\varphi(t)=\left\langle I'(tu),tu\right\rangle \text{ on } (0,+\infty)\).
It follows from (\ref{a12}) and Lemma \ref{lem 6.6} that there exist positive constants
\(0<\xi_2<\frac{a}{K_2^2}\) and \(C_{\xi_2}>0\) such that
\begin{eqnarray*}
\varphi(t)
&=& at^2\|u\|^2+bt^4\|u\|^4-\int_{\Omega}f(tu)tu\,d\mu  \\
&\geq&
at^2\|u\|^2+bt^4\|u\|^4
-\xi_2t^2\int_{\Omega}|u|^2\,d\mu
-C_{\xi_2}t^p\int_{\Omega}|u|^p\,d\mu  \\
&=&
at^2\|u\|^2+bt^4\|u\|^4
-\xi_2t^2\|u\|_{L^2(\Omega)}^2
-C_{\xi_2}t^p\|u\|_{L^p(\Omega)}^p  \\
&\geq&
at^2\|u\|^2+bt^4\|u\|^4
-\xi_2K_2^2t^2\|u\|^2
-C_{\xi_2}K_p^pt^p\|u\|^p  \\
&=&
(a-\xi_2K_2^2)t^2\|u\|^2+bt^4\|u\|^4
-C_{\xi_2}K_p^pt^p\|u\|^p .
\end{eqnarray*}
Then, since \(a-\xi_2K_2^2>0\) and \(p\geq4\), we have \(\varphi(t)>0\) for sufficiently small \(t>0\).

On the other hand, by applying \((F_3)\) together with (\ref{c2}), we obtain
\begin{eqnarray*}
\varphi(t)
&=& at^2\|u\|^2+bt^4\|u\|^4-\int_{\Omega}f(tu)tu\,d\mu  \\
&=& at^2\|u\|^2+bt^4\|u\|^4
-lt^4\int_{\Omega}|u|^4\,d\mu-\int_{\Omega} f_{\infty}(tu)tu\,d\mu  \\
\end{eqnarray*}
and
\begin{align*}
\limsup_{|t| \to \infty} \frac{\varphi(t)}{t^4} &= \limsup_{|t| \to \infty} \left[ \frac{a\|u\|^2}{t^2} + b\|u\|^4 - l \int_\Omega |u|^4 d\mu - \frac{\int_{\Omega} f_{\infty}(tu)tu\,d\mu}{t^4} \right] \\
&= b\|u\|^4 - l \int_\Omega |u|^4 d\mu < 0.
\end{align*}
Then, it follows from \(u\in G\) that there exists a sufficiently large \(t>0\) such that \(\varphi(t)<0\). By the continuity of \(\varphi(t)\) on \((0,+\infty)\), there exists \(t_u>0\) satisfying
$\varphi(t_u)=0$, which implies that \(t_uu\in\mathcal N\).

Next, we prove the uniqueness of $t_u$. Assume, by contradiction, that there exist $u\in G$ and two distinct positive real numbers $t_1,t_2$ satisfying
$t_1\neq t_2$, $t_1u\in\mathcal{N}$, and $t_2u\in\mathcal{N}$.
By Corollary \ref{cor3.5}, taking $u:=t_1u$ and
$t:=\frac{t_2}{t_1}$, we have
\[
I(t_1u)
\geq
I(t_2u)
+
\frac{a(1-\theta_0)}{4}
\left[1-\left(\frac{t_2}{t_1}\right)^2\right]^2
t_1^2\|u\|^2
>
I(t_2u).
\]
Similarly, by taking \(u:=t_2u\) and \(t:=\dfrac{t_1}{t_2}\), we obtain that
\[
I(t_2u)
\geq
I(t_1u)
+
\frac{a(1-\theta_0)}{4}
\left[1-\left(\frac{t_1}{t_2}\right)^2\right]^2
t_2^2\|u\|^2
>
I(t_1u).
\]
This contradiction shows that \(t_1=t_2\). Therefore, the positive number \(t_u\) satisfying \(t_uu\in\mathcal N\) is unique.
\end{proof}
\begin{lemma}\label{lem3.17}
{\it Suppose that \((F_1)-(F_4)\) hold. Then}
\[
\inf_{u\in\mathcal N}I(u):=c
=
\inf_{u\in G}\max_{t\geq0}I(tu).
\]
\end{lemma}

\begin{proof}[\bf{Proof}]
On the one hand, it follows from Corollary \ref{cor3.7} and Corollary \ref{cor3.15} that

\begin{equation}\label{a23}
\inf_{u\in G}\max_{t\geq0}I(tu)
\leq
\inf_{u\in\mathcal N}\max_{t\geq0}I(tu)
=
\inf_{u\in\mathcal N}I(u)
=c .
\end{equation}
On the other hand, for any \(u\in G\), Lemma \ref{lem3.16} implies
\[
\max_{t\geq0}I(tu)
\geq
I(t_u u)
\geq
\inf_{u\in\mathcal N}I(u)
=c .
\]
Therefore,
\begin{equation}\label{a24}
\inf_{u\in G}\max_{t\geq0}I(tu)
\geq
\inf_{u\in\mathcal N}I(u)
=c .
\end{equation}
Combining (\ref{a23}) and (\ref{a24}), we conclude the result.
\end{proof}

\begin{lemma}\label{lem3.18}
{\it Suppose that \((F_1)-(F_4)\) hold. Then  \(m>0\) and \(c>0\) can be achieved.}
\end{lemma}
\begin{proof}[\bf{Proof}]
By Lemma \ref{lem3.11}, for any $u\in E$, there exists a unique pair
$(s_u,t_u)\in(0,+\infty)\times(0,+\infty)$ such that
$s_u u^+ + t_u u^-\in\mathcal M$. Hence, \(\mathcal{M}\neq\emptyset\). Let $\{u_n\}\subset \mathcal{M}$ be such that $I(u_n)\rightarrow m$. By (\ref{b3}), (\ref{b4}) and Lemma \ref{lem3.13}, it holds that
\begin{eqnarray}\label{h1}
m+o(1)
&= &I(u_n)-\frac{1}{4}\langle I'(u_n),u_n\rangle \notag\\
&= &\frac{a}{2}\|u_n\|^2+\frac{b}{4}\|u_n\|^4
-\int_{\Omega}F(u_n)\,d\mu
-\frac{a}{4}\|u_n\|^2
-\frac{b}{4}\|u_n\|^4
+\frac{1}{4}\int_{\Omega}f(u_n)u_n\,d\mu \notag\\
&= &\frac{a}{4}\|u_n\|^2
+\frac{1}{4}\int_{\Omega}f(u_n)u_n\,d\mu
-\int_{\Omega}F(u_n)\,d\mu \notag\\
&\geq&
\frac{a}{4}\|u_n\|^2
-\frac{a\theta_0\lambda_1}{4}
\int_{\Omega}|u_n|^2\,d\mu \notag\\
&\geq&
\frac{a}{4}\|u_n\|^2
-\frac{a\theta_0}{4}\|u_n\|^2 \notag\\
&=&
\frac{a(1-\theta_0)}{4}\|u_n\|^2 .
\end{eqnarray}
This implies that the sequence $\{u_n\}$ is bounded in $\mathcal{H}^{1,2}_0(\Omega)$. Hence, there exists a constant $M>0$ such that $\|u_n\|\le M$. Since $\mathcal{H}^{1,2}_0(\Omega)$ is finite-dimensional, Lemma \ref{lem 6.6} yields a subsequence, still denoted by $\{u_n\}$, and a function $\tilde{u}\in \mathcal{H}^{1,2}_0(\Omega)$ such that $u_n\rightarrow\tilde{u}$ in $\mathcal{H}^{1,2}_0(\Omega)$. Furthermore, by Lemma \ref{lem 6.6}, we have
\begin{eqnarray}\label{a26}
\begin{cases}
u_n^\pm \to \tilde{u}^\pm &\text{in}\;\mathcal{H}^{1,2}_0(\Omega),\\
u_n^\pm(x)\to \tilde{u}^\pm(x) & \mbox{for each } x\in \Omega,\\
u_n(x)\to \tilde{u}(x) & \mbox{for each } x\in \Omega,\\
u_n^\pm\rightarrow\tilde{u}^\pm & \text{in} \; L^s(\Omega), s\in[2,+\infty).
\end{cases}
\end{eqnarray}
Since $\{u_n\}\subset \mathcal{M}$, it follows from the definition of \(\mathcal{M}\) that $\langle I'(u_n),u_n^\pm\rangle=0$. Thus, Proposition 3.2 in \cite{Ou 2024} yields
\begin{align}\label{a27}
a\|u_n^{\pm}\|^2+b\|u_n^{\pm}\|^4
={}&
\int_{\Omega} f(u_n^\pm)u_n^\pm\,d\mu
-\frac{a}{2}\mathcal W_{u_n}
-b\|u_n^+\|^2\|u_n^-\|^2
-\frac{3b}{2}\|u_n^\pm\|^2\mathcal W_{u_n}
-\frac{b}{2}\|u_n^\mp\|^2\mathcal W_{u_n}
\notag\\
&-\frac{b}{2}
\left[
\sum_{x\in\Omega^-}
\sum_{\substack{y\thicksim x\\y\in\Omega^+}}
w_{xy}u_n^-(x)u_n^+(y)
\right]^2
-\frac{b}{2}
\left[
\sum_{x\in\Omega^+}
\sum_{\substack{y\thicksim x\\y\in\Omega^-}}
w_{xy}u_n^-(y)u_n^+(x)
\right]^2
\notag\\
&-b
\sum_{x\in\Omega^+}
\sum_{\substack{y\thicksim x\\y\in\Omega^-}}
w_{xy}u_n^-(y)u_n^+(x)
\sum_{x\in\Omega^-}
\sum_{\substack{y\thicksim x\\y\in\Omega^+}}
w_{xy}u_n^-(x)u_n^+(y).
\end{align}
It follows from (\ref{a12}) and (\ref{a27}) that there exist $\varepsilon_3\in (0,\frac{a}{K_2^2})$ and a constant $C_{\varepsilon_3}>0$ such that
\begin{align}\label{a28}
a\|u_n^\pm\|^2
&\leq a\|u_n^\pm\|^2+b\|u_n^\pm\|^4
\notag\\
&=
\int_{\Omega}f(u_n^\pm)u_n^\pm\,d\mu
-\frac{a}{2}\mathcal W_{u_n}
-b\|u_n^+\|^2\|u_n^-\|^2
-\frac{3b}{2}\|u_n^\pm\|^2\mathcal W_{u_n}
-\frac{b}{2}\|u_n^\mp\|^2\mathcal W_{u_n}
\notag\\
&\quad
-\frac{b}{2}
\left[
\sum_{x\in\Omega^-}
\sum_{\substack{y\thicksim x\\y\in\Omega^+}}
w_{xy}u_n^-(x)u_n^+(y)
\right]^2
-\frac{b}{2}
\left[
\sum_{x\in\Omega^+}
\sum_{\substack{y\thicksim x\\y\in\Omega^-}}
w_{xy}u_n^-(y)u_n^+(x)
\right]^2
\notag\\
&\quad
-b
\sum_{x\in\Omega^+}
\sum_{\substack{y\thicksim x\\y\in\Omega^-}}
w_{xy}u_n^-(y)u_n^+(x)
\sum_{x\in\Omega^-}
\sum_{\substack{y\thicksim x\\y\in\Omega^+}}
w_{xy}u_n^-(x)u_n^+(y)
\notag\\
&\leq
\varepsilon_3\|u_n^\pm\|_{L^{2}(\Omega)}^{2}
+C_{\varepsilon_3}\|u_n^\pm\|_{L^{p}(\Omega)}^{p}
\notag\\
&\leq
\varepsilon_3K_2^2\|u_n^\pm\|^{2}
+C_{\varepsilon_3}K_p^p\|u_n^\pm\|^{p}.
\end{align}
Thus, $\|u_n^\pm\|^{p-2}\geq \frac{a-\varepsilon_3 K_2^2}{C_{\varepsilon_3}K_p^p}>0$, which, together with (\ref{a26}), implies that $\tilde{u}^\pm\neq 0$. Notice that the integral of a function $u$ is defined as in (\ref{eq4}), and $\Omega$ consists of finitely many points. Then, by (\ref{b4}), (\ref{a26}), and (\ref{a27}), we obtain
\begin{align}\label{a29}
&a\|\tilde u^\pm\|^2+b\|\tilde u^\pm\|^4
\notag\\
={}&
\lim_{n\rightarrow\infty}a\|u_n^\pm\|^2
+\lim_{n\rightarrow\infty}b\|u_n^\pm\|^4
\notag\\
={}&
\lim_{n\rightarrow\infty}
\Bigg\{
\int_{\Omega} f(u_n^\pm)u_n^\pm\,d\mu
-\frac{a}{2}\mathcal W_{u_n}
-b\|u_n^+\|^2\|u_n^-\|^2
-\frac{3b}{2}\|u_n^\pm\|^2\mathcal W_{u_n}
-\frac{b}{2}\|u_n^\mp\|^2\mathcal W_{u_n}
\notag\\
&\qquad
-\frac{b}{2}
\left[
\sum_{x\in\Omega^-}
\sum_{\substack{y\sim x\\y\in\Omega^+}}
w_{xy}u_n^-(x)u_n^+(y)
\right]^2
-\frac{b}{2}
\left[
\sum_{x\in\Omega^+}
\sum_{\substack{y\sim x\\y\in\Omega^-}}
w_{xy}u_n^-(y)u_n^+(x)
\right]^2
\notag\\
&\qquad
-b
\sum_{x\in\Omega^+}
\sum_{\substack{y\sim x\\y\in\Omega^-}}
w_{xy}u_n^-(y)u_n^+(x)
\sum_{x\in\Omega^-}
\sum_{\substack{y\sim x\\y\in\Omega^+}}
w_{xy}u_n^-(x)u_n^+(y)
\Bigg\}
\notag\\
={}&
\int_{\Omega}f(\tilde u^\pm)\tilde u^\pm\,d\mu
-\frac{a}{2}\mathcal W_{\tilde u}
-b\|\tilde u^+\|^2\|\tilde u^-\|^2
-\frac{3b}{2}\|\tilde u^\pm\|^2\mathcal W_{\tilde u}
-\frac{b}{2}\|\tilde u^\mp\|^2\mathcal W_{\tilde u}
\notag\\
&\qquad
-\frac{b}{2}
\left[
\sum_{x\in\Omega^-}
\sum_{\substack{y\sim x\\y\in\Omega^+}}
w_{xy}\tilde u^-(x)\tilde u^+(y)
\right]^2
-\frac{b}{2}
\left[
\sum_{x\in\Omega^+}
\sum_{\substack{y\sim x\\y\in\Omega^-}}
w_{xy}\tilde u^-(y)\tilde u^+(x)
\right]^2
\notag\\
&\qquad
-b
\sum_{x\in\Omega^+}
\sum_{\substack{y\sim x\\y\in\Omega^-}}
w_{xy}\tilde u^-(y)\tilde u^+(x)
\sum_{x\in\Omega^-}
\sum_{\substack{y\sim x\\y\in\Omega^+}}
w_{xy}\tilde u^-(x)\tilde u^+(y).
\end{align}
Therefore, $\langle I'(\tilde u),\tilde u^\pm\rangle=0$, and together with
$\tilde u^\pm\neq0$, we obtain $\tilde u\in\mathcal M$. By taking $\tilde{s}=0$ and $\tilde{t}=0$ in Lemma \ref{lem3.2}, we have
\[
m=I(\tilde{u})
\geq
\frac{a(1-\theta_0)}{4}\|\tilde{u}^{+}\|^2
+
\frac{a(1-\theta_0)}{4}\|\tilde{u}^{-}\|^2
=
\frac{a(1-\theta_0)}{4}
\left(
\|\tilde{u}^{+}\|^2+\|\tilde{u}^{-}\|^2
\right)>0 .
\]

Similarly, we can prove that \(c>0\) is achieved. By Lemma \ref{lem3.16}, for any $u \in G$, there exists a unique \(t_u>0\) such that \(t_uu\in\mathcal N\). Therefore, \(\mathcal{N}\neq\emptyset\).
Let $\{u_n\}\subset \mathcal{N}$ be such that $I(u_n)\rightarrow c$. By a similar argument to that used in (\ref{h1}), we also obtain
\begin{equation}\label{h2}
c+o(1)=I(u_n)-\frac{1}{4}\langle I'(u_n),u_n\rangle
\geq
\frac{a(1-\theta_0)}{4}\|u_n\|^2 .
\end{equation}
This shows that the sequence $\{u_n\}$ is bounded in $\mathcal{H}^{1,2}_0(\Omega)$. Therefore, there exists a constant $C>0$ such that $\|u_n\|\le C$. Since $\mathcal{H}^{1,2}_0(\Omega)$ is finite-dimensional,  Lemma \ref{lem 6.6} allows us to extract a subsequence, still denoted by $\{u_n\}$, and there exists a function $\bar{u}\in \mathcal{H}^{1,2}_0(\Omega)$ such that
\begin{align}\label{ww2}
\left\{
\begin{aligned}
& u_n\rightarrow \bar{u} \mbox{\ in } \mathcal{H}^{1,2}_0(\Omega);\\
& u_n(x) \rightarrow \bar{u}(x)\ \ \mbox{for all }x\in \Omega.
\end{aligned}
\right.
\end{align}
Since $\{u_n\}\subset \mathcal{N}$, the definition of $\mathcal{N}$ gives $\langle I'(u_n),u_n\rangle=0$.  Consequently, (\ref{b4}) yields
\begin{align}\label{w2}
a\|u_n\|^2+b\|u_n\|^4
={}&
\int_{\Omega} f(u_n)u_n\,d\mu.
\end{align}
It follows from (\ref{a12}) and (\ref{w2}) that one can choose \(\varepsilon_4\in(0,\frac{a}{K_2^2})\) and a positive constant \(C_{\varepsilon_4}\) such that
\begin{align}\label{w3}
a\|u_n\|^2
&\leq a\|u_n\|^2+b\|u_n\|^4
\notag\\
&=
\int_{\Omega}f(u_n)u_n\,d\mu
\notag\\
&\leq
\varepsilon_4\|u_n\|_{L^{2}(\Omega)}^{2}
+C_{\varepsilon_4}\|u_n\|_{L^{p}(\Omega)}^{p}
\notag\\
&\leq
\varepsilon_4K_2^2\|u_n\|^{2}
+C_{\varepsilon_4}K_p^p\|u_n\|^{p}.
\end{align}
Thus, $\|u_n\|^{p-2}\geq \frac{a-\varepsilon_4 K_2^2}{C_{\varepsilon_4}K_p^p}>0$. In view of the pointwise convergence in (\ref{ww2}), we have $\bar{u}\neq 0$. Since \(\Omega\) consists of finitely many points, the integral in (\ref{eq4}) reduces to a finite sum. Consequently, by (\ref{b4}), (\ref{ww2}),  and (\ref{w2}), we obtain
\begin{align}\label{w4}
a\|\bar u\|^2+b\|\bar u\|^4
&=\lim_{n\rightarrow\infty}a\|u_n\|^2
+\lim_{n\rightarrow\infty}b\|u_n\|^4
=\lim_{n\rightarrow\infty}\int_{\Omega}f(u_n)u_n\,d\mu
=\int_{\Omega}f(\bar u)\bar u\,d\mu .
\end{align}
Consequently, $\langle I'(\bar u),\bar u\rangle=0$, which yields
$\bar u\in\mathcal N$. Note that $\bar{u}\neq0$. By taking $t=0$ in Lemma \ref{lem3.3}, we obtain
\[
c=I(\bar {u})
\geq
\frac{a(1-\theta_0)}{4}
\|\bar {u}\|^2
>0.
\]
The proof is complete.
\end{proof}

\begin{lemma}\label{lem3.19}
{\it Suppose that \((F_1)-(F_4)\) hold. The minimizers $\tilde{u}$ and $\bar{u}$ of \(\inf_{\mathcal{M}} I\) and \(\inf_{\mathcal{N}} I\) are critical points of \(I\).}
\end{lemma}
\begin{proof}[{\bf Proof}]
 We only need to note that $W_{\tilde u}>0$ in Lemma 3.1. Then the proof is the same as Lemma 3.9 in \cite{Ou 2024}. We omit the detail proofs.
\end{proof}
\vskip2mm
\noindent
\normalfont
{\bf Proof of Theorem 1.1.}
By Lemmas \ref{lem3.18} and \ref{lem3.19}, there exist
$\tilde{u}\in\mathcal{M}$ and $\bar{u}\in\mathcal{N}$ such that
$I(\tilde{u})=m$, $I'(\tilde{u})=0$, and
$I(\bar{u})=c$, $I'(\bar{u})=0$.
Moreover, it follows from Corollary \ref{cor3.6}, Lemma \ref{lem3.16}
and Proposition 3.2 in \cite{Ou 2024} that there exist positive constants
$s_u$ and $t_u$ satisfying
$s_u\tilde{u}^+\in\mathcal{N}$ and $t_u\tilde{u}^-\in\mathcal{N}$.
Consequently, we have
\begin{eqnarray*}
        m
&  =   &  I(\tilde{u})=\max_{s,t\geq0} I(s\tilde{u}^++t\tilde{u}^-)\\
& > &  \max_{s,t\geq0} \big[I(s\tilde{u}^+)+I(t\tilde{u}^-)\big]\\
&  =   &  \max_{s\geq0} I(s\tilde{u}^+)+\max_{t\geq0} I(t\tilde{u}^-)\\
& \geq &  I(s_u\tilde{u}^+)+I(t_u\tilde{u}^-)\\
& \geq &  I(\bar{u})+I(\bar{u})\\
& \geq &  2c>0.
\end{eqnarray*}
This completes the proof of Theorem 1.1.
\normalfont
\vskip3mm
\noindent{\bf Funding information}

\noindent
This work is supported by Scientific Research Fund of Yunnan Provincial Department of Education in China (grant No. 2025J0085), Philosophy and Social Sciences Academic Excellence Cultivation Program of Kunming University of Science and Technology in China (grant No. JPSC2025006),  Yunnan Fundamental Research Projects in China (grant No: 202501AS070170) and Yunnan Provincial Innovation Guidance and Technology-Based Enterprise Cultivation Program in China (grant No: 202404CC110017), and Interdisciplinary Research Special Program of Kunming University of Science and Technology (grant No. KUST-xk2026018).

\vskip3mm
\noindent{\bf Conflict of interest}

\noindent
On behalf of all authors, the corresponding author states that there is no conflict of interest.

\vskip3mm
\noindent{\bf Authors' contribution}

\noindent
The authors contribute to the main manuscript equally.

\vskip2mm
\renewcommand\refname{References}
\normalfont

\end{document}